\documentclass[10pt,a4paper,english]{scrartcl}
\usepackage{lmodern}

\usepackage[T1]{fontenc}
\usepackage[utf8]{inputenc}
\usepackage{mathtools}
\usepackage{amsmath}
\usepackage{amsthm}
\usepackage{amssymb}

\makeatletter

\theoremstyle{plain}
\newtheorem{thm}{\protect\theoremname}
\theoremstyle{plain}
\newtheorem{lem}[thm]{\protect\lemmaname}
\theoremstyle{remark}
\newtheorem{rem}[thm]{\protect\remarkname}

\@ifundefined{date}{}{\date{}}
\newenvironment{thmbis}[1]
  {\renewcommand{\thethm}{\ref{#1}$'$}%
   \addtocounter{thm}{-1}%
   \begin{thm}}
  {\end{thm}}

\makeatother

\usepackage{babel}
\providecommand{\lemmaname}{Lemma}
\providecommand{\remarkname}{Remark}
\providecommand{\theoremname}{Theorem}

\begin{document}
\global\long\def\R{\mathbb{R}}%
\global\long\def\vu{\sqrt{1+|\nabla u|^{2}}}%
\global\long\def\divergence{\mathop{\mathrm{div}}}%
\global\long\def\graph{\mathop{\mathrm{graph}}}%
\global\long\def\L{\mathop{\mathcal{L}}}%
\global\long\def\ddt{\frac{\mathrm{d}}{\mathrm{d}t}}%

\title{$\boldsymbol{\mathrm{H^{\alpha}}}$-flow of mean convex, complete
graphical hypersurfaces}
\author{Wolfgang Maurer\thanks{funded by the Deutsche Forschungsgemeinschaft (DFG, German Research Foundation) Project number 336454636}}
\maketitle
\begin{abstract}
We consider the evolution of hypersurfaces in $\R^{n+1}$ with normal
velocity given by a positive power of the mean curvature. The hypersurfaces
under consideration are assumed to be strictly mean convex (positive
mean curvature), complete, and given as the graph of a function. Long-time
existence of the $H^{\alpha}$-flow is established by means of approximation
by bounded problems.
\end{abstract}

\section{Introduction}

Let $\alpha>0$. The flow by the $\alpha^{\text{th}}$ power of the
mean curvature, or short $H^{\alpha}$-flow, is the evolution of a
hypersurface $M_{t}$, such that at each point the normal velocity
equals $H^{\alpha}$, the $\alpha^{\text{th}}$ power of the mean
curvature $H$. The hypersurface is assumed to be strictly mean convex
($H>0$) at all times. If $X(\cdot,t)\colon M^{n}\to\mathbb{R}^{n+1}$
are (local) embeddings of the time-dependent hypersurface $M_{t}$,
then $H^{\alpha}$-flow is described by the equation ($\nu$ is the
normal vector)
\begin{equation}
\langle\dot{X},\nu\rangle=H^{\alpha}\,.\label{eq H^alpha-flow}
\end{equation}

In the case of $\alpha=1$, one obtains the mean curvature flow. The
mean curvature flow has been studied extensively and continues to
be an active area of research. Among all curvature functions one can
impose for the normal speed the mean curvature certainly is outstanding
in its importance and comparable simplicity. A very influential paper
was \cite{Huisken} where Huisken has shown that mean curvature flow
shrinks convex hypersurfaces to ``round points.'' Similar results
have been proven for various other normal speeds of homogeneity one
before Schulze investigated the $H^{\alpha}$-flow, which has homogeneity
$\alpha$, in \cite{Schulze-Evolution} (also cf. \cite{Schulze-Convexity,BCD,AW}).
The $H^{\alpha}$-flow is a step away from homogeneity one but still
has a fairly simple structure.

In this paper we are concerned with graphical hypersurfaces. For the
mean curvature flow of entire graphs long-time existence has been
established in \cite{EH}. Franzen generalized this to $H^{\alpha}$-flow
of entire graphs in \cite{Franzen} under some technical assumptions.
Sáez and Schnürer generalized the mean curvature flow of entire graphs
in a different direction: In \cite{SS} they considered complete graphs.
In contrast to entire graphs these need not be defined over all of
$\R^{n}$ and are instead defined over an open subset. To represent
a complete hypersurface, the graph representation must tend to infinity
at the boundary of this subset. It turns out that in this case, too,
no singularities occur on a finite level, where they would be visible.
Instead, singularities form at infinity in some sense.

In this paper we bring together these two trends of generalization
and prove long-time existence for the $H^{\alpha}$-flow of complete
graphs. The following is the main result.
\begin{thm}
\label{thm intro}Let $\alpha>0$. Let $\Omega_{0}\subset\mathbb{R}^{n}$
be open. Let $u_{0}\colon\Omega_{0}\to\mathbb{R}$ be smooth and such
that $\graph u_{0}$ is of positive mean curvature. Furthermore, we
suppose that the sets $\{x\in\R^{n}:u_{0}(x)\le a\}$ are compact
for any $a\in\mathbb{R}$.

Then there exists an $H^{\alpha}$-flow $(u,\Omega)$ of complete
graphs with initial value $(u_{0},\Omega_{0})$ (cf.\ Theorem \ref{thm main}
for a precise formulation).
\end{thm}

At this point it is mandatory to say more about the existing literature
and on how this result fits in. Concerning entire graphs, we have
already mentioned \cite{Franzen}, where the $H^{\alpha}$-flow is
considered, too, and the results of which we could improve by dropping
a certain condition called ``$\nu$-condition'' as well as expanding
the result to complete graphs. Still concerning entire graphs, the
author wants to mention \cite{Holland}, where flows with speed $S_{k}^{1/k}$
with $S_{k}$ an elementary symmetric polynomial of the principal
curvatures are considered, and \cite{AS}, where general curvature
flows of homogeneity one are investigated. Inspired by the mean curvature
flow of complete graphs there is a number of papers that are concerned
with the flow of complete graphs by various normal speeds. In \cite{Xiao},
Xiao provides the needed a priori estimates for general curvature
flows of homogeneity one. In two papers \cite{CD,CDKL} the existence
of the flow of complete graphs by powers of the Gaussian curvature
and the $Q_{k}$-flow of convex, complete graphs is established. In
\cite{LiLv}, the flow by powers of general curvature functions is
considered for complete graphs that are convex. While the $H^{\alpha}$-flow
that is considered in the present paper is fully non-linear and not
of homogeneity one, it is still subsumed in the latter paper \cite{LiLv}.
However, in the present paper we do not demand that the hypersurfaces
be convex; we only assume mean convexity which is the natural assumption
in this case. This does not only make the a priori estimates more
difficult but also significantly changes the approximation scheme.
In the convex case one usually cuts off at some height and reflects
at that height to get a convex closed hypersurface. These usually
behave nicely under geometric flows. In particular the two symmetric
parts stay graphical. This idea does not work anymore without convexity.

Having already touched upon the proof, let us briefly summarize the
main ideas. To established long-time existence for the $H^{\alpha}$-flow
of complete graphs, we approximate the non-compact problem by compact
ones. To this end, we cut off the initial complete graph at increasing
heights. These will then serve as initial data for initial boundary
value problems with Dirichlet boundary condition. These approximation
problems can then be solved using parabolic PDE-theory in Hölder-spaces.
Instead of $H^{\alpha}$-flow the author found it easier to consider
$H^{\alpha(x)}$-flow for these problems. The reason is that estimates
at the boundary are difficult for the fully nonlinear flow and changing
the equation towards the boundary to the quasilinear mean curvature
flow ($\alpha=1$) facilitates the problem. For being able to extract
a limit from the approximation problems, it is crucial to have local
a priori estimates. These are gained through the maximum principle
applied to a suitable test function. Here it suffices to have bounds
that are local in height and one may utilize the height function as
a localization function. With these it is possible to pass to a limit
using a variation on the Arzelà-Ascoli theorem from \cite{SS}.

The paper is organized as follows. After some preliminaries about
our sign-conventions, different parametrizations for the $H^{\alpha}$-flow,
and evolution equations, we introduce and solve the approximation
problems in Section \ref{sec:Auxiliary-problems}. Section \ref{sec:A-priori-estimates}
is devoted to the needed a priori estimates that are local in height.
Finally, in Section \ref{sec:Complete-hypersurfaces} the main theorem
is proven.

The author is grateful to Oliver Schnürer for suggesting looking into
the topic and for discussions. The author also likes to thank Benjamin
Lambert for starting a joint project on the topic and sharing his
prolific thoughts.

\section{Preliminaries}

\subsection{Sign convention}

Let $X(\cdot,t)\colon M^{n}\to\R^{n+1}$ be a family of immersions
with parameter $t$ which we interpret as time. Our convention for
the normal $\nu$ is that it points in the same direction as the mean
curvature vector $\vec{H}=\Delta X$ ($\Delta$ denotes the Laplace-Beltrami
operator with respect to the induced metric). This is a well-defined
choice of normal because the hypersurfaces are assumed to be strictly
mean convex. Accordingly, the second fundamental form is defined by
$h_{ij}=\left\langle X_{ij},\,\nu\right\rangle $ and the mean curvature
is given by $H=g^{ij}\,h_{ij}$, where $(g^{ij})$ is the inverse
of the metric $(g_{ij})$ induced by $X(\cdot,t)$. For example, if
$M_{t}$ bounds a convex domain, then $(h_{ij})$ is non-negative
and $\nu$ points into the interior.

\subsection{Parametrizations}

Equation (\ref{eq H^alpha-flow}) is invariant under time-dependent
reparametrizations. So there is still a certain (large) degree of
freedom in that equation. This can be fixed by fully prescribing $\dot{X}$
instead of only its normal component. Of course there are many different
ways to do this, which we refer to as different parametrizations.
We will use two different ones, the parametrization along the normal
and the graphical parametrization. In a parametrization with $\dot{X}$
pointing along the normal direction equation (\ref{eq H^alpha-flow})
becomes 
\begin{equation}
\dot{X}=H^{\alpha}\,\nu\,.\label{eq H^alpha-flow nu}
\end{equation}
If the evolving hypersurface is described by graphs of functions $u(\cdot,t)$
with the assumption $\left\langle \nu,e_{n+1}\right\rangle >0$ (with
$e_{n+1}=(0,\ldots,0,1)$) and if we use a graphical parametrization,
i.e., $X(x,t)=(x,u(x,t))$, equation (\ref{eq H^alpha-flow}) is equivalent
to 
\begin{equation}
\frac{\partial_{t}u}{\vu}=\left(\divergence\left(\frac{\nabla u}{\vu}\right)\right)^{\alpha}\equiv H^{\alpha}\,,\label{eq H^alpha-flow graph}
\end{equation}
where we have used
\[
w\coloneqq\left\langle \nu,e_{n+1}\right\rangle \equiv\frac{1}{\vu}\,.
\]

The parametrization along the normal (\ref{eq H^alpha-flow nu}) is
geometrically appealing and yields comprehensible evolution equations
for geometric quantities. We shall use it for deriving the needed
estimates. The graphical parametrization (\ref{eq H^alpha-flow graph})
is preferable from a PDE point of view. It is a scalar parabolic differential
equation and we can harness the power of the theory of those equations.
Higher estimates and existence of solutions are our takings.

\subsection{Evolution equations}

For the estimates we rely on the maximum principle. It is applied
in conjunction with the linear parabolic operator 
\[
\L\coloneqq\frac{\partial_{t}}{\alpha(X)\,H^{\alpha(X)-1}}-\Delta_{M_{t}}\,.
\]
This operator is defined on a hypersurface that evolves by $H^{\alpha(X)}$-flow
with parametrization along the normal, i.~e.\ , the immersions $X(\cdot,t)$
satisfy (\ref{eq H^alpha-flow nu}). The position-dependence of the
exponent is needed in the approximation problems and therefore incorporated.
For the a priori estimates of Section \ref{sec:A-priori-estimates}
we do not need the dependence on the position.

We will apply the operator $\L$ to different geometric quantities.
The outcome is summarized in the following lemma.
\begin{lem}
The evolution equations for $X$, $\nu$, $H^{\alpha(X)}$, and $|A|^{2}$
are given by (we suppress the dependence of $\alpha$ on $X$ in the
notation)
\begin{align}
\L X & =\left(\frac{1}{\alpha}-1\right)\,H\nu\,,\label{eq: ee X}\\
\L\nu & =|A|^{2}\,\nu-\frac{1}{\alpha}\,H\log(H)\,X_{i}^{\beta}\,\alpha_{\beta}\,g^{ij}\,X_{j}\,,\label{eq: ee nu}\\
\L H^{\alpha} & =\left(|A|^{2}+\frac{1}{\alpha}\,H\log(H)\,\partial_{\nu}\alpha\right)H^{\alpha}\,,\label{eq: ee H^alpha}\\
\L|A|^{2} & =-2\,|\nabla A|^{2}+2\,(\alpha-1)\,H^{-1}\,h^{ij}\,\nabla_{i}H\,\nabla_{j}H+2\,|A|^{4}+2\,\frac{1-\alpha}{\alpha}\,H\,h_{j}^{i}\,h_{k}^{j}\,h_{i}^{k}\label{eq: ee |A|^2}\\
 & \quad+4\left(\log(H)+\frac{1}{\alpha}\right)\alpha_{\beta}\,X_{j}^{\beta}\,h^{ij}\,\nabla_{j}H\nonumber \\
 & \quad+2\frac{1}{\alpha}\,H\log(H)\,\left(\left(\alpha_{\beta\gamma}+\log(H)\,\alpha_{\beta}\,\alpha_{\gamma}\right)X_{i}^{\beta}\,X_{j}^{\gamma}\,h^{ij}+\partial_{\nu}\alpha\,|A|^{2}\right)\,.\nonumber 
\end{align}
\end{lem}

\begin{proof}
The first follows directly from $\dot{X}=H^{\alpha}\,\nu$ and $\Delta X=H\,\nu$.

For the normal $\nu$, we notice that the condition $\left\langle \nu,\nu\right\rangle =1$
implies $\left\langle \dot{\nu},\nu\right\rangle =0$ and $\left\langle \nu,\nu_{i}\right\rangle =0$.
Applying a time derivative to $\left\langle \nu,X_{i}\right\rangle =0$
we obtain
\begin{align*}
\left\langle \dot{\nu},X_{i}\right\rangle  & =-\left\langle \nu,\dot{X}_{i}\right\rangle =-\left\langle \nu,\left(H^{\alpha(X)}\,\nu\right)_{i}\right\rangle =-\left(H^{\alpha(X)}\right)_{i}\\
 & =-\alpha\,H^{\alpha-1}H_{i}-\log(H)\,H^{\alpha}\,\alpha_{\beta}\,X_{i}^{\beta}\,.
\end{align*}
Therefore, we conclude 
\[
\dot{\nu}=-\left(\alpha\,H^{\alpha-1}H_{i}+\log(H)\,H^{\alpha}\,\alpha_{\beta}\,X_{i}^{\beta}\right)g^{ij}X_{j}\,.
\]
By Weingarten's equation and Codazzi's equation there holds
\[
\Delta\nu=g^{ij}\,\nu_{ij}=-g^{ij}\,(h_{i}^{k}\,X_{k})_{j}=-\left(\nabla^{k}H\right)X_{k}-h_{i}^{k}\,h_{k}^{i}\,\nu\,.
\]
Combing the last two yields the assertion.

Before we continue with the evolution equation for $H^{\alpha(X)}$,
we consider
\begin{align}
\ddt g_{ij} & =\ddt\left\langle X_{i},X_{j}\right\rangle =\left\langle \left(H^{\alpha(X)}\,\nu\right)_{i},X_{j}\right\rangle +\left\langle X_{i},\left(H^{\alpha(X)}\,\nu\right)_{j}\right\rangle \nonumber \\
 & =0-H^{\alpha}\,h_{ij}+0-H^{\alpha}\,h_{ij}=-2H^{\alpha}\,h_{ij}\,,\nonumber \\
\ddt g^{ij} & =-g^{ik}\left(\ddt g_{kl}\right)g^{lj}=2H^{\alpha}\,h^{ij}\,,\nonumber \\
\frac{\mathrm{d}}{\mathrm{d}t}h_{ij} & =-\ddt\left\langle \nu_{j},X_{i}\right\rangle =-\left\langle \dot{\nu}_{j},X_{i}\right\rangle -\left\langle \nu_{j},\dot{X}_{i}\right\rangle \nonumber \\
 & =\left\langle \left(\nabla^{k}\left(H^{\alpha(X)}\right)X_{k}\right)_{j},X_{i}\right\rangle +\left\langle h_{j}^{k}\,X_{k},\left(H^{\alpha(X)}\,\nu\right)_{i}\right\rangle \nonumber \\
 & =\nabla_{j}\nabla_{i}H^{\alpha(X)}+0+0-H^{\alpha}\,h_{j}^{k}\,h_{ik}\,.\label{eq:h_ij dot}
\end{align}
Now we are in position:
\begin{align*}
\ddt H^{\alpha(X)} & =\alpha\,H^{\alpha-1}\left(g^{ij}\left(\ddt h_{ij}\right)+\left(\ddt g^{ij}\right)h_{ij}\right)+H^{\alpha}\log(H)\,\alpha_{\beta}\ddt X^{\beta}\\
 & =\alpha\,H^{\alpha-1}\left(\Delta H^{\alpha(X)}-H^{\alpha}\,|A|^{2}+2H^{\alpha}\,|A|^{2}\right)+H^{\alpha}\log(H)\,\alpha_{\beta}\,H^{\alpha}\,\nu^{\beta}\\
 & =\alpha\,H^{\alpha-1}\left(\Delta H^{\alpha(X)}+H^{\alpha}|A|^{2}+\frac{1}{\alpha}\,H^{\alpha+1}\log(H)\,\partial_{\nu}\alpha\right)\,.
\end{align*}

Next we compute the evolution equation of $h_{ij}$. To this end we
compute the term 
\begin{align*}
\nabla_{j}\nabla_{i}H^{\alpha(X)} & =\nabla_{j}\left(\alpha(X)\,H^{\alpha(X)-1}\,\nabla_{i}H+H^{\alpha(X)}\log H\,\alpha_{\beta}(X)\,X_{i}^{\beta}\right)\\
 & =\alpha\,H^{\alpha-1}\,\nabla_{j}\nabla_{i}H+\alpha_{\beta}\,X_{j}^{\beta}\,H^{\alpha-1}\,\nabla_{i}H\\
 & \quad+\alpha\,(\alpha-1)\,H^{\alpha-2}\,\nabla_{j}H\,\nabla_{i}H+\alpha\,H^{\alpha-1}\log(H)\,\alpha_{\beta}\,X_{j}^{\beta}\,\nabla_{i}H\\
 & \quad+\alpha\,H^{\alpha-1}\,\nabla_{j}H\log(H)\,\alpha_{\beta}\,X_{i}^{\beta}+H^{\alpha}\left(\log H\right)^{2}\,\alpha_{\gamma}\,X_{j}^{\gamma}\,\alpha_{\beta}\,X_{i}^{\beta}\\
 & \quad+H^{\alpha-1}\nabla_{j}H\,\alpha_{\beta}\,X_{i}^{\beta}+H^{\alpha}\log(H)\,\alpha_{\beta\gamma}\,X_{i}^{\beta}\,X_{j}^{\gamma}\\
 & \quad+H^{\alpha}\log(H)\,\alpha_{\beta}\,\nu^{\beta}\,h_{ij}\,.
\end{align*}
From (\ref{eq:h_ij dot}) we infer
\begin{equation}
\begin{split}\L h_{ij} & =-g^{kl}\,\nabla_{l}\nabla_{k}h_{ij}+\nabla_{j}\nabla_{i}H+(\alpha-1)\,H^{-1}\,\nabla_{i}H\,\nabla_{j}H\\
 & \quad+\left(\log(H)+\frac{1}{\alpha}\right)\left(\alpha_{\beta}\,X_{j}^{\beta}\,\nabla_{i}H+\alpha_{\beta}\,X_{i}^{\beta}\,\nabla_{j}H\right)\\
 & \quad+\frac{1}{\alpha}\,H\log(H)\,\left(\left(\alpha_{\beta\gamma}+\log(H)\,\alpha_{\beta}\,\alpha_{\gamma}\right)X_{i}^{\beta}\,X_{j}^{\gamma}+\partial_{\nu}\alpha\,h_{ij}\right)\\
 & \quad-\frac{1}{\alpha}H\,h_{j}^{k}\,h_{ik}\,.
\end{split}
\label{eq: h_ij no cancel}
\end{equation}
To cancel the second derivatives of $h_{ij}$ on the right hand side,
we will need the following identity:
\begin{align*}
\nabla_{j}\nabla_{i}h_{kl} & =\nabla_{j}\nabla_{l}h_{ik}=\nabla_{l}\nabla_{j}h_{ik}+{R_{jli}}^{a}\,h_{ak}+{R_{jlk}}^{a}\,h_{ia}\\
 & =\nabla_{l}\nabla_{j}h_{ik}+\left(h_{ij}\,h_{l}^{a}-h_{j}^{a}\,h_{li}\right)h_{ak}+\left(h_{jk}\,h_{l}^{a}-h_{j}^{a}\,h_{lk}\right)h_{ia}\\
 & =\nabla_{l}\nabla_{k}h_{ij}+h_{ij}\,h_{k}^{a}\,h_{al}-h_{il}\,h_{j}^{a}\,h_{ak}+h_{jk}\,h_{i}^{a}\,h_{al}-h_{kl}\,h_{i}^{a}\,h_{aj}\,.
\end{align*}
We apply this in the following way
\begin{align*}
\nabla_{j}\nabla_{i}H & =\nabla_{j}\nabla_{i}(g^{kl}\,h_{kl})=g^{kl}\,\nabla_{j}\nabla_{i}h_{kl}=g^{kl}\,\nabla_{l}\nabla_{k}h_{ij}+|A|^{2}\,h_{ij}-Hh_{i}^{a}\,h_{aj}\,.
\end{align*}
Inserting this into (\ref{eq: h_ij no cancel}) yields 
\[
\begin{split}\L h_{ij} & =|A|^{2}\,h_{ij}-\left(1+\frac{1}{\alpha}\right)H\,h_{i}^{k}\,h_{kj}+(\alpha-1)\,H^{-1}\,\nabla_{i}H\,\nabla_{j}H\\
 & \quad+\left(\log(H)+\frac{1}{\alpha}\right)\left(\alpha_{\beta}\,X_{j}^{\beta}\,\nabla_{i}H+\alpha_{\beta}\,X_{i}^{\beta}\,\nabla_{j}H\right)\\
 & \quad+\frac{1}{\alpha}\,H\log(H)\,\left(\left(\alpha_{\beta\gamma}+\log(H)\,\alpha_{\beta}\,\alpha_{\gamma}\right)X_{i}^{\beta}\,X_{j}^{\gamma}+\partial_{\nu}\alpha\,h_{ij}\right)\,.
\end{split}
\]

Finally, we turn to the evolution equation for $|A|^{2}=g^{ik}\,g^{jl}\,h_{ij}\,h_{kl}$.
\begin{align*}
\L|A|^{2} & =\frac{1}{\alpha}\,H\,\left(2\,h^{ik}\,g^{jl}\,h_{ij}\,h_{kl}+2\,g^{ik}\,h^{jl}\,h_{ij}\,h_{kl}\right)+\left(\L h_{ij}\right)h^{ij}+h^{kl}\,\L h_{kl}-2|\nabla A|^{2}\\
 & =2\,h^{ij}\left(\L h_{ij}\right)+4\frac{1}{\alpha}H\,h_{j}^{i}\,h_{k}^{j}\,h_{i}^{k}-2|\nabla A|^{2}\\
 & =-2\,|\nabla A|^{2}+2\,(\alpha-1)\,H^{-1}\,h^{ij}\,\nabla_{i}H\,\nabla_{j}H+2\,|A|^{4}+2\,\frac{1-\alpha}{\alpha}\,H\,h_{j}^{i}\,h_{k}^{j}\,h_{i}^{k}\\
 & \quad+2\,h^{ij}\left(\log(H)+\frac{1}{\alpha}\right)\left(\alpha_{\beta}\,X_{j}^{\beta}\,\nabla_{i}H+\alpha_{\beta}\,X_{i}^{\beta}\,\nabla_{j}H\right)\\
 & \quad+2\,h^{ij}\,\frac{1}{\alpha}\,H\log(H)\,\left(\left(\alpha_{\beta\gamma}+\log(H)\,\alpha_{\beta}\,\alpha_{\gamma}\right)X_{i}^{\beta}\,X_{j}^{\gamma}+\partial_{\nu}\alpha\,h_{ij}\right)\,.\qedhere
\end{align*}
\end{proof}

\section{Auxiliary problems\label{sec:Auxiliary-problems}}

The aim of this section is to provide the solutions of the approximation
problems. In a first attempt, one may wish to solve the initial boundary
value problem for the $H^{\alpha}$-flow ((\ref{eq H^alpha-flow graph}))
\begin{equation}
\begin{cases}
\partial_{t}u=\frac{1}{w}\,H^{\alpha}=\vu\left(\divergence\left(\frac{\nabla u}{\vu}\right)\right)^{\alpha} & \text{on }Q\times(0,T)\,,\\
u(x,t)=0 & \text{for }x\in\partial Q,\;t\in[0,T]\,,\\
u(\cdot,0)=u_{0}\,,
\end{cases}\label{eq ap wish}
\end{equation}
for $Q\subset\mathbb{R}^{n}$ open and bounded, $T>0$, and for (smooth)
initial data $u_{0}$ with $u_{0}|_{\partial Q}\equiv0$ and such
that the mean curvature of $\graph u_{0}$ is positive. However, this
initial boundary value problem is problematic for a number of reasons.
Firstly, the right hand side of the equation must vanish at the boundary.
So the mean curvature vanishes at the boundary. This destroys the
uniform parabolicity of the equation (unless $\alpha=1$). Secondly,
compatibility conditions are only satisfied for initial data $u_{0}$
with specific behavior near the boundary. On the more technical side,
we will need to prove $C^{2}$-estimates at the boundary for this
fully nonlinear equation. These are hard to work out for geometric
flows, especially if the homogeneity is different from one.

To bypass these problems we consider a different auxiliary problem.
Eventually, all we need of our auxiliary problem is that it has a
given $u_{0}$ as initial condition and that, below a certain given
height, the equation describes $H^{\alpha}$-flow. This can be accomplished
by a number of different initial boundary value problems and we do
not need to insist on (\ref{eq ap wish}). The route from (\ref{eq ap wish})
to a new initial boundary value problem that is easier to solve is
described now. To ensure the parabolicity of the equation at the boundary,
we introduce time dependent boundary values $u(x,t)=c\,t$ for $(x,t)\in\partial Q\times[0,T]$.
Moreover, we add a term to the right hand side of the equation that
depends on $u_{0}$ and will ensure compatibility conditions of any
order. Lastly, we make the equation quasilinear near the boundary
by introducing a position dependence for the exponent $\alpha$ such
that $\alpha\equiv1$ near the boundary.

We consider the initial boundary value problem 
\begin{equation}
\begin{cases}
\dot{u}=G(\nabla^{2}u,\nabla u,x)-\xi(x)\,(G_{0}(x)-c) & \text{on }Q\times(0,T)\,,\\
u(x,t)=c\,t & \text{for }x\in\partial Q,\;t\in[0,T]\,,\\
u(\cdot,0)=u_{0}\;.
\end{cases}\label{eq ap ap}
\end{equation}
We suppose: 
\begin{itemize}
\item $T>0$, and $Q\subset\mathbb{R}^{n}$ is a smooth, bounded, and open
domain. 
\item The operator is given by 
\[
\begin{split}G(\nabla^{2}u,\nabla u,x) & =\frac{1}{w(\nabla u)}H(\nabla^{2}u,\nabla u)^{\alpha(x)}\\
 & =\vu\left(\divergence\left(\frac{\nabla u}{\vu}\right)\right)^{\alpha(x)}\;.
\end{split}
\]
So the equation $\dot{u}=G(\nabla^{2}u,\nabla u,x)$ describes $H^{\alpha(x)}$-flow
for $\graph u(\cdot,t)$. 
\item $\alpha(x)$ is a smooth function on $Q$ with values in $(0,\infty)$
and which is constant $\alpha\equiv1$ in a neighborhood of the boundary
$\partial Q$. 
\item The initial datum $u_{0}$ is smooth on $\overline{Q}$ and satisfies
$u_{0}\le0$, $u_{0}|_{\partial Q}\equiv0$, and its graph has positive
mean curvature $H(\nabla^{2}u_{0},\nabla u_{0})>0$ on $\overline{Q}$. 
\item The constant $c>0$ is chosen in such a way that $G_{0}(x)\coloneqq G(\nabla^{2}u_{0},\nabla u_{0},x)\ge c$
holds. 
\item The function $\xi(x)$ is a smooth cut-off function with values in
$[0,1]$. We demand $\xi\equiv1$ in a neighborhood of the boundary
$\partial Q$ and that $\alpha\equiv1$ on a neighborhood of the support
of $\xi$. 
\end{itemize}
\begin{lem}
\label{lem ap} The initial boundary value problem (\ref{eq ap ap})
has a unique smooth solution, even for $T=\infty$. 
\end{lem}

\begin{proof}
The problem satisfies compatibility conditions of any order because
if one substitutes $u_{0}$ for $u$, then the right hand side of
the equation becomes equal to the constant $c$ in a neighborhood
of the boundary. This is perfectly compatible with the boundary condition
$u(x,t)=c\,t$ for any order of differentiation. For the sake of clarity,
we explicitly check the first compatibility conditions. For the zeroth
compatibility condition we fix $x\in\partial Q$ and we put $t=0$
in the second line of (\ref{eq ap ap}), which yields $u(x,0)=0$.
The third line in (\ref{eq ap ap}) yields $u(x,0)=u_{0}(x)=0$ because
$u_{0}\equiv0$ on $\partial Q$. So the two lines are compatible.
For the first compatibility condition one has to check the equation
(first line) where the time derivative is computed from the second
line and the spatial derivatives are derived from the third line.
The second line gives $\dot{u}(x,0)=c$. The third line ($u(\cdot,0)=u_{0}$)
yields 
\[
G(\nabla^{2}u,\nabla u,x)-\xi(x)(G_{0}(x)-c)=(1-\xi(x))G_{0}(x)+\xi(x)\,c=c
\]
at $t=0$ and for $x$ in the neighborhood of the boundary where $\xi\equiv1$.
In particular the first compatibility condition is satisfied. For
the second compatibility condition one takes a time derivative in
the equation. This gives 
\[
\ddot{u}=G_{r_{ij}}(\nabla^{2}u,\nabla u,x)\,\dot{u}_{ij}+G_{p_{i}}(\nabla^{2}u,\nabla u,x)\,\dot{u}_{i}\,.
\]
But we have already seen that $\dot{u}(\cdot,0)\equiv c$ where $\xi\equiv1$
when computed from $u_{0}$ using the equation. Therefore, the spatial
derivatives of $\dot{u}$ on the right-hand side vanish. As $\ddot{u}$
also vanishes when computed from $u(x,t)=c\,t$, the second compatibility
condition holds true too. The higher order conditions are similar.

To investigate the parabolicity properties of the equation, we differentiate
the right hand side with respect to $\nabla^{2}u$: 
\begin{equation}
\begin{split}\frac{\partial}{\partial r_{ij}}\bigg(G(r,p,x)-\xi(x)(G_{0}(x)-c)\bigg) & =\frac{1}{w(p)}\frac{\partial}{\partial r_{ij}}(H(r,p))^{\alpha(x)}\\
 & =\frac{\alpha\,H(r,p)^{\alpha(x)-1}}{w(p)}\frac{\partial}{\partial r_{ij}}H(r,p)\\
 & =\frac{\alpha\,H(r,p)^{\alpha(x)-1}}{w(p)}\left(\delta^{ij}-\frac{p^{i}\,p^{j}}{1+|p|^{2}}\right)
\end{split}
\label{eq: ap parabolicity}
\end{equation}
This positive definite matrix has both-sided bounds on its eigenvalues
if $p=\nabla u$ is bounded and there is a both-sided bound $C^{-1}\le H\le C$
for the mean curvature.

We are clearly going to need some a priori estimates. Let $u$ be
a solution of (\ref{eq ap ap}) with some $T>0$ which is not necessarily
that of the assertion of the lemma. Our goal now is to work through
the following program. 
\begin{enumerate}
\item $c\le\dot{u}\le\sup G_{0}$ \label{it dotu} 
\item $u_{0}(x)\le u(x,t)-c\,t\le0$ \label{it u} 
\item $\nabla u$ is uniformly bounded at the boundary $\partial Q$ \label{it nabla u partial Q} 
\item upper bound on $H^{\alpha(x)}$ \label{it ub H} 
\item interior gradient bound \label{it nabla u} 
\item positive lower bound on $H^{\alpha(x)}$ \label{it lb H} 
\item bound for the $C^{2;1}$-norm in a neighborhood of $\partial Q$ \label{it norm partial Q} 
\item estimation of $|A|^{2}$ corresponding to $\graph u$. \label{it |A|^2} 
\end{enumerate}

\paragraph{(\ref{it dotu})}

By differentiating the equation (\ref{eq ap ap}), we obtain 
\begin{equation}
\frac{\partial}{\partial t}\dot{u}=G_{r_{ij}}\dot{u}_{ij}+G_{p_{i}}\dot{u}_{i}\;,\label{eq ap dotu}
\end{equation}
a linear parabolic equation for $\dot{u}$. On the lateral boundary
$\partial Q\times[0,T]$, the boundary condition yields $\dot{u}\equiv c$.
On the bottom part $Q\times\{0\}$, we have $\dot{u}=(1-\xi)\,G_{0}+\xi\,c$.
So $\dot{u}$ attains values in the range $[c,\sup G_{0}]$ on the
parabolic boundary. According to the parabolic maximum principle,
$\dot{u}$ attains its minimum and maximum on the parabolic boundary.
Hence, we have $c\le\dot{u}\le\sup G_{0}$ on all of $\overline{Q}\times[0,T]$.

\paragraph{(\ref{it u})}

The inequality $c\le\dot{u}$ implies $u_{0}\le u(x,t)-c\,t$.

For $\varepsilon>0$, we consider the function $\overline{u}(x,t)\coloneqq c\,t+\varepsilon$.
On the parabolic boundary holds $u<\overline{u}$. Assume that at
some point $u(x,t)=\overline{u}(x,t)$ holds for the first time. Then
$\graph u(\cdot,t)$ touches a horizontal hyperplane from below. This
contradicts the hypothesis that $\graph u(\cdot,t)$ is mean convex.
So $u(x,t)<\overline{u}(x,t)$ holds on all of $\overline{Q}\times[0,T]$.
Because $\varepsilon>0$ is arbitrary, we infer $u\le c\,t$.

\paragraph{(\ref{it nabla u partial Q})}

Because of $u_{0}|_{\partial Q}\equiv0$, (\ref{it u}) immediately
induces the boundary gradient estimate $|\nabla u(x,t)|=|\nabla(u(x,t)-c\,t)|\le|\nabla u_{0}(x)|$
for $(x,t)\in\partial Q\times[0,T]$.

\paragraph{(\ref{it ub H})}

An upper bound for $H^{\alpha(x)}$ follows from (\ref{it dotu}):
\[
\begin{split}H^{\alpha(x)} & =w\,G=w\,\big(\dot{u}+\xi(G_{0}-c)\big)\le1\cdot\big(\sup G_{0}+1\cdot(G_{0}-0)\big)\\
 & \le2\,\sup G_{0}\;.
\end{split}
\]

\paragraph{(\ref{it nabla u})}

Taking a spatial derivative in (\ref{eq ap ap}) yields 
\begin{equation}
\dot{u}_{k}=G_{r_{ij}}\,u_{kij}+G_{p_{i}}\,u_{ki}+G_{x^{k}}-\xi_{k}\,(G_{0}-c)-\xi\,\partial_{k}G_{0}.\label{eq ap equation_k}
\end{equation}
We shall show that $G_{x^{k}}$ is bounded in a controlled way. It
holds 
\[
G_{x^{k}}=\frac{\partial}{\partial x^{k}}\,\frac{1}{w}\,H^{\alpha(x)}=\frac{1}{w}\,\log(H)\,H^{\alpha(x)}\,\frac{\partial\alpha}{\partial x}\;.
\]
Because $a^{\alpha}\,\log a\to0$ as $a\to0$ for any $\alpha>0$,
an upper bound for $|G_{x^{k}}|$ follows from the upper bound (\ref{it ub H})
for $H^{\alpha(x)}$ (of course the (least) upper bound for $\alpha$
enters too).

Let $C>0$ be a constant depending only on the data such that $|G_{x^{k}}-\xi_{k}\,(G_{0}-c)-\xi\,\partial_{k}G_{0}|<C$.
The maximum principle applied to equation (\ref{eq ap equation_k})
implies that $u_{k}-C\,t$ attains its maximum on the parabolic boundary
of $Q\times[0,T]$. In the same way $u_{k}+C\,t$ attains its minimum
on the parabolic boundary. There, on the parabolic boundary, we have
$|u_{k}|\le|\nabla u_{0}|$ by (\ref{it nabla u partial Q}). So we
obtain a ($T$-dependent) bound for $|\nabla u|$ on all of $Q\times[0,T]$.

\paragraph{(\ref{it lb H})}

From (\ref{it nabla u}) we have a controlled positive lower bound
for $w>0$. On the other hand, we have the lower bound $\dot{u}\ge c$
from (\ref{it dotu}). Combined, these give a lower bound for $H^{\alpha}(x)$:
\[
H^{\alpha(x)}=w\,\big(\dot{u}+\xi\,(G_{0}-c)\big)\ge w\,\dot{u}\ge c\,\inf w>0\;.
\]

\paragraph{(\ref{it norm partial Q})}

In a neighborhood of $\partial Q$, there holds $\alpha\equiv1$.
There, the equation is quasilinear as the right hand side of (\ref{eq: ap parabolicity})
does not depend on $r$ for $\alpha\equiv1$. From general regularity
theory of parabolic equations \cite[Chapter 6]{LSU}, we obtain from
$C^{1;1}$-estimates ((\ref{it dotu}),(\ref{it u}),(\ref{it nabla u}))
and estimates on the parabolicity constants, which follow from (\ref{it ub H}),(\ref{it nabla u}),(\ref{it lb H}),
local higher order estimates: in a first step $C^{1+\beta;\frac{1+\beta}{2}}$-estimates
and then all higher order estimates from the theory of Schauder. This
establishes local a priori estimates for all derivatives of $u$ in
some neighborhood of $\partial Q\times[0,T]$.

\paragraph{(\ref{it |A|^2})}

Now we turn to the estimation of $|A|^{2}$. We do calculations on
the time-dependent hypersurface which is given by the graph of $u(\cdot,t)$
and we choose a parametrization $X$ such that $\partial_{t}X$ points
in normal direction (cf.\ (\ref{eq H^alpha-flow nu})). Slightly
abusing notation, we extend $\alpha$ to a function on $\mathbb{R}^{n+1}$
by setting $\alpha(x^{1},\ldots,x^{n},x^{n+1})\equiv\alpha(x^{1},\ldots,x^{n})$.

We consider the quantity $f\coloneqq\log|A|^{2}-p\log(H^{\alpha(X)}-b)$
for a constant $b>0$ such that $b\le\frac{1}{2}H^{\alpha}\le b^{-1}$
and a constant $p>0$ to be chosen later. Finding a controlled upper
bound for $|A|^{2}$ is equivalent to finding a controlled upper bound
for $f$. The reason being that we have already established both-sided
bounds for $H^{\alpha(X)}$, i.e., we have control on $b$, and consequently
we have both-sided control on the second term in the expression for
$f$. So let $(p_{0},t_{0})$ be a maximal point of $f$. Our goal
is to prove that $f$ is controllably bounded at $(p_{0},t_{0})$.

If $(X^{1},\ldots,X^{n})|_{(p_{0},t_{0})}$ is in the support of $\xi$
then $f(p_{0},t_{0})$ is bounded by virtue of (\ref{it norm partial Q}).
Furthermore, $f(p_{0},t_{0})$ is controlled by $u_{0}$ if $t_{0}=0$.
Therefore, we may assume that $\big((X^{1},\ldots,X^{n})(p_{0},t_{0}),t_{0}\big)\in\left(Q\times(0,T]\right)\setminus\left(\mathrm{supp}\,\xi\times[0,T]\right)$
holds.

Because $p_{0}$ is an interior point, the maximality condition implies
that the first derivative vanishes. It follows 
\begin{equation}
\frac{\nabla|A|^{2}}{|A|^{2}}=p\,\frac{\nabla H^{\alpha(X)}}{H^{\alpha(X)}-b}\qquad\text{at }(p_{0},t_{0})\,.\label{eq apce nabla eq 0}
\end{equation}

We consider the differential operator 
\begin{equation}
\L\coloneqq\frac{\partial_{t}}{\alpha(X)\,H^{\alpha(X)-1}}-\Delta\,.\label{eq apce L}
\end{equation}
Here, $\Delta$ denotes the Laplace-Beltrami-Operator of the hypersurface
given by the graph of $u(\cdot,t)$.

Outside of the support of $\xi$, $u$ solves the graphical $H^{\alpha(X)}$-flow.
In particular, we may apply Lemma \ref{lem evol eq}. Because the
point $(p_{0},t_{0})$ is a parabolically interior point and because
of its maximality condition for $f=\log|A|^{2}-p\log(H^{\alpha}-b)$,
at this point holds 
\begin{equation}
\begin{split}0 & \le\L\left(\log|A|^{2}-p\log(H^{\alpha}-b)\right)\\
 & =\frac{\L|A|^{2}}{|A|^{2}}+\left|\frac{\nabla|A|^{2}}{|A|^{2}}\right|^{2}-p\,\frac{\L\big(H^{\alpha}-b\big)}{H^{\alpha}-b}-p\left|\frac{\nabla H^{\alpha}}{H^{\alpha}-b}\right|^{2}\\
 & =\frac{\L|A|^{2}}{|A|^{2}}-p\,\frac{\L\big(H^{\alpha}-b\big)}{H^{\alpha}-b}+\big(p^{2}-p\big)\left|\frac{\nabla H^{\alpha}}{H^{\alpha}-b}\right|^{2}
\end{split}
\label{eq: apce 0 le L}
\end{equation}
due to (\ref{eq apce nabla eq 0}). From (\ref{eq: ee H^alpha}),
(\ref{eq: ee |A|^2}), and (\ref{eq: apce 0 le L}) we obtain 
\begin{equation}
\begin{split}0\le & -2\,\frac{|\nabla A|^{2}}{|A|^{2}}+2\,|\alpha-1|\,\frac{|\nabla H|^{2}}{|A|\,H}\\
 & +4\,\frac{1}{\alpha}\,(1+\alpha\,|\log H|)\,\frac{1}{|A|}\,|\mathrm{D}\alpha|\,|\nabla H|+\big(p^{2}-p\big)\left|\frac{\nabla H^{\alpha}}{H^{\alpha}-b}\right|^{2}\\
 & +\left(2-p\,\frac{H^{\alpha}}{H^{\alpha}-b}\right)\left(|A|^{2}+\frac{1}{\alpha}\,H\log(H)\,\partial_{\nu}\alpha\right)+2\,\frac{|1-\alpha|}{\alpha}\,H\,|A|\\
 & +\frac{2}{\alpha}\,\frac{H}{|A|}\,|\log(H)|\,\left(|\mathrm{D}^{2}\alpha|+|\log H|\,|\mathrm{D}\alpha|^{2}\right)\;.
\end{split}
\label{eq: apce i}
\end{equation}

We have already established upper and lower bounds for $H$. The exponent-function
$\alpha\colon\mathbb{R}^{n+1}\to\mathbb{R}$ is fixed and it is, together
with its derivatives, $\mathrm{D}\alpha$ and $\mathrm{D}^{2}\alpha$,
to be seen as controlled. Moreover, we may assume for an arbitrary,
but fixed, $\delta>0$ that $H\le\delta\,|A|$. Otherwise, $|A|$
is controlled by $H$ (and $\delta$) and the desired curvature estimate
would follow. Therefore, we conclude from (\ref{eq: apce i}) for
a universal constant $C>0$ 
\begin{equation}
\begin{split}0\le & -2\,\frac{|\nabla A|^{2}}{|A|^{2}}+2\,|\alpha-1|\,\delta\,\frac{|\nabla H|^{2}}{H^{2}}\\
 & +4\,\frac{1}{\alpha}\,(1+\alpha\,|\log H|)\,\frac{1}{|A|}\,|\mathrm{D}\alpha|\,|\nabla H|+\big(p^{2}-p\big)\left|\frac{\nabla H^{\alpha}}{H^{\alpha}-b}\right|^{2}\\
 & +\left(2-p\,\frac{H^{\alpha}}{H^{\alpha}-b}+2\,\frac{|1-\alpha|}{\alpha}\,\delta\right)|A|^{2}+C\;.
\end{split}
\label{eq: apce after delta}
\end{equation}
Thereon we apply the following results: 
\begin{align*}
 & |\nabla|A|^{2}|^{2}=|2\,h^{ij}\,\nabla h_{ij}|^{2}\le4\,|A|^{2}\,|\nabla A|^{2}\\
 & \Rightarrow-2\,\frac{|\nabla A|^{2}}{|A|^{2}}\le-\frac{1}{2}\left|\frac{\nabla|A|^{2}}{|A|^{2}}\right|^{2}=-\frac{p^{2}}{2}\left|\frac{\nabla H^{\alpha}}{H^{\alpha}-b}\right|^{2}\,,\\
 & \frac{\nabla H^{\alpha}}{H^{\alpha}}=\nabla\log H^{\alpha}=\alpha\,\nabla\log H+\log(H)\,\nabla\alpha=\alpha\,\frac{\nabla H}{H}+\log(H)\,\nabla\alpha\\
 & \Rightarrow\frac{|\nabla H|^{2}}{H^{2}}\le C\left|\frac{\nabla H^{\alpha}}{H^{\alpha}}\right|^{2}+C\le C\,\left|\frac{\nabla H^{\alpha}}{H^{\alpha}-b}\right|^{2}+C\,,\\
 & \frac{1}{|A|}\,|\mathrm{D}\alpha|\,|\nabla H|\le C\,|\mathrm{D}\alpha|^{2}+\frac{|\nabla H|^{2}}{|A|^{2}}\le C+\delta^{2}\,\frac{|\nabla H|^{2}}{H^{2}}\le C+\delta^{2}\,C\left|\frac{\nabla H^{\alpha}}{H^{\alpha}-b}\right|^{2}\,.
\end{align*}
With these three inequalities we obtain from (\ref{eq: apce after delta})
\begin{equation}
\begin{split}0\le & \left(-\frac{1}{2}p^{2}+p^{2}-p+2\,|\alpha-1|\,\delta\,C+\delta^{2}\,C\right)\left|\frac{\nabla H^{\alpha}}{H^{\alpha}-b}\right|^{2}\\
 & +\left(2-p\,\frac{H^{\alpha}}{H^{\alpha}-b}+2\,\frac{|1-\alpha|}{\alpha}\,\delta\right)|A|^{2}+C\;.
\end{split}
\label{eq: apce after 3 ineq}
\end{equation}
Now we set $p$ to $p=2-\frac{b^{2}}{2}$. Then, using $b\le\frac{1}{2}H^{\alpha}\le b^{-1}$,
\[
\begin{split}2-p\,\frac{H^{\alpha}}{H^{\alpha}-b} & =\frac{(2-p)H^{\alpha}-2b}{H^{\alpha}-b}=\frac{\frac{b^{2}}{2}H^{\alpha}-2b}{H^{\alpha}-b}\\
 & \le\frac{b-2b}{H^{\alpha}-b}=\frac{-b}{H^{\alpha}-b}\le\frac{-b}{2b^{-1}}=-\frac{b^{2}}{2}\;.
\end{split}
\]
Because of $p<2$, $-p+\frac{1}{2}p^{2}<0$ holds. If we choose $\delta>0$
sufficiently small, but still in a controlled fashion, we can conclude
from (\ref{eq: apce after 3 ineq}) 
\[
0\le0\cdot\left|\frac{\nabla H^{\alpha}}{H^{\alpha}-b}\right|^{2}-\frac{b^{2}}{4}\cdot|A|^{2}+C\;.
\]
This demonstrates that $|A|^{2}\le C$ holds at the point $(p_{0},t_{0})$.
Of course, this implies that $f(p_{0},t_{0})$ is bounded (with control).
To show this was our goal.

\paragraph{Conclusion of the proof of Lemma \ref{lem ap}.}

The estimates (\ref{it nabla u}), (\ref{it ub H}), and (\ref{it lb H})
ensure that the equation is uniformly parabolic along any solution
with a priori bounds on the parabolicity constants.

Following the short time existence proof in \cite{Gerhardt} yields
a smooth solution to (\ref{eq ap ap}) for some time. The same argument
shows that for the maximal time interval $[0,T)$ the estimates for
$u(\cdot,t)$ degenerate as $t\to T$. To prove long time existence
($T=\infty$), it remains to exclude this behavior.

We have already proven $C^{2;1}$-estimates for $u$. Using these,
we can infer a $C^{\beta;\frac{\beta}{2}}$-estimate for $\dot{u}$
from (\ref{eq ap dotu}) through Krylov-Safonov estimates: We can
apply in our case the interior estimate of \cite{TW}. (Steps 1 and
2 of their proof are sufficient for our purpose. These consist of
applying the linear Krylov-Safonov estimate to equation (\ref{eq ap dotu})
and then using theory for, in our case, quasilinear elliptic equations
for spatial $C^{2,\beta}$-estimates on $u(\cdot,t)$. In a second
step, classical elliptic Schauder theory, applied to $(t_{1}-t_{0})^{-\beta/4}\,(u(\cdot,t_{1})-u(\cdot,t_{0}))$,
yields Hölder estimates in time for $\nabla^{2}u(\cdot,t)$. Cf.\ \cite{TW}
for more details.) In a neighborhood of the boundary, estimates follow
from \cite{LSU} because the equation is quasilinear there, as we
have discussed above in (\ref{it norm partial Q}).

With the $C^{2+\beta;\frac{2+\beta}{2}}$ estimates for $u$, we can
start a bootstrapping argument using Schauder estimates and conclude
estimates for any derivatives of $u$. Thus, we have achieved long
time existence.

Uniqueness of the solution follows from the maximum principle.
\end{proof}

\section{A priori estimates\label{sec:A-priori-estimates}}

We consider strictly mean convex, graphical hypersurfaces $M_{t}=\graph u(\cdot,t)$
in this section. We assume that $u\ge0$ so that $M_{t}$ lies in
the upper half-space $\{x\in\mathbb{R}^{n+1}\colon x^{n+1}\ge0\}$.
For a fixed chosen height $a\in\mathbb{R}$, we suppose that $\{x\colon u(x,t)\le a\}$
is compact for all times $t$ and that $u$ solves (\ref{eq H^alpha-flow graph})
on $\{(x,t)\colon u(x,t)<a+1\}$, i.e., the part of $M_{t}$ below
height $a+1$ moves by the $H^{\alpha}$-flow. How $M_{t}$ behaves
above that height is irrelevant because we will use a cut-off function
which vanishes above height $a$.

Let $X(\cdot,t)$ be a parametrization of $M_{t}$ with a suitable
(time dependent) parameter space such that (\ref{eq H^alpha-flow nu})
holds. We will use this parametrization throughout this section.

We prove the estimates via a maximum principle type argument. The
general strategy pursued here is to multiply the quantity to be estimated
with a localization function. From there, a test function is constructed.
If one applies a parabolic operator to a function, then the result
has a certain sign at parabolically interior, minimal or maximal points
of the function. By using a suitable parabolic operator this can be
exploited to bound the test function if it has been chosen adequately.
In a last step one recovers from the bound on the test function a
local estimate for the quantity in question. A suitable linear parabolic
differential operator is given in our case by 
\begin{equation}
\L\coloneqq\frac{1}{\alpha\,H^{\alpha-1}}\,\partial_{t}-\Delta_{M_{t}}\,.
\end{equation}

We will again use 
\begin{equation}
w\coloneqq\left\langle \nu,e_{n+1}\right\rangle .
\end{equation}

\begin{lem}
\label{lem evol eq} On $\{(p,t)\colon X^{n+1}(p,t)\le a\}$ hold
\begin{align}
\L X^{n+1} & =\frac{1-\alpha}{\alpha}\,H\,w\,,\label{eq evol w}\\
\L w & =|A|^{2}\,w\,,\\
\L H^{\alpha} & =|A|^{2}\,H^{\alpha}\,,\\
\begin{split}\L|A|^{2} & =-2\,|\nabla A|^{2}+2\,(\alpha-1)\,H^{-1}\,h^{kl}\,\nabla_{k}H\,\nabla_{l}H\\
 & \quad+2\,|A|^{4}+2\,\frac{1-\alpha}{\alpha}\,H\,h_{j}^{i}\,h_{k}^{j}\,h_{i}^{k}\,.
\end{split}
\label{eq evol |A|^2}
\end{align}
\end{lem}

\begin{proof}
Follows directly from Lemma \ref{lem evol eq}.
\end{proof}
For the computations the following identity for functions $\varphi>0$
is helpful 
\begin{equation}
\begin{split}\L\log\varphi & =\frac{1}{\alpha\,H^{\alpha-1}}\,\frac{\partial_{t}\varphi}{\varphi}-g^{ij}\,\nabla_{i}\frac{\nabla_{j}\varphi}{\varphi}=\frac{1}{\alpha\,H^{\alpha-1}}\,\frac{\partial_{t}\varphi}{\varphi}-g^{ij}\left(\frac{\nabla_{i}\nabla_{j}\varphi}{\varphi}-\frac{\nabla_{i}\varphi\,\nabla_{j}\varphi}{\varphi^{2}}\right)\\
 & =\frac{\L\varphi}{\varphi}+\left|\frac{\nabla\varphi}{\varphi}\right|^{2}\,.
\end{split}
\label{eq evol log}
\end{equation}

Before we start estimating, we still have to define the cut-off function.
For a parameter $b\ge0$ and $a$ from above it is of the form 
\begin{equation}
\psi\coloneqq\left(a-b\,t-X^{n+1}\right)_{+}\,.
\end{equation}
By assumption, $\{p\colon\psi(p,t)>0\}$ is compact for any $t\ge0$.
Where $\psi>0$, 
\begin{equation}
\L\log\psi=\frac{\L\psi}{\psi}+\left|\frac{\nabla\psi}{\psi}\right|^{2}=-\frac{1-\alpha}{\alpha}\,H\,w\,\psi^{-1}-\frac{b}{\alpha\,H^{\alpha-1}}\,\psi^{-1}+\left|\frac{\nabla\psi}{\psi}\right|^{2}\label{eq evol psi}
\end{equation}
holds.

The usage of the height function $X^{n+1}$ is the most frequently
applied localization method in the study of curvature flows without
singularities. The idea to add the time dependent term $b\,t$ has
been adopted from \cite{CD,CDKL}.

\paragraph{Gradient bound.}

Because of $w\equiv(1+|\nabla u|^{2})^{-1/2}$, an upper bound for
$|\nabla u|$ is equivalent to a lower positive bound for $w$. We
have 
\begin{equation}
\psi^{-1}\,w\ge\begin{cases}
\inf\limits _{t=0}\psi^{-1}\,w & \alpha\le1\,,\\
\min\left\{ \inf\limits _{t=0}\psi^{-1}\,w,\;\frac{\alpha}{n\,(\alpha-1)},\;\frac{b}{a\,(\alpha-1)}\right\}  & \alpha>1\,.
\end{cases}\label{eq gradient bound}
\end{equation}

\begin{proof}
We work on the set $\{\psi>0\}$. By assumption, $\{p\colon X^{n+1}(p,t)\le a-b\,t\}$
is compact for any $t$. Hence, $f\coloneqq\log\psi^{-1}\,w=\log w-\log\psi$
attains a minimum on any compact time interval $[0,T]$. It suffices
to prove (\ref{eq gradient bound}) for times in $[0,T]$ if $T$
is kept arbitrary. Let the minimum be attained at $(p_{0},t_{0})$.
Note that $(p_{0},t_{0})$ is also a minimum point of $\psi^{-1}\,w$.
Because of this, it suffices to prove (\ref{eq gradient bound}) at
this point $(p_{0},t_{0})$.

If $t_{0}=0$, the assertion (\ref{eq gradient bound}) follows.

If $t_{0}>0$ and taking into account that $f\to\infty$ as $p\to\partial\{p\colon X^{n+1}(p,t)\le a-b\,t\}$,
and hence that $p_{0}$ is in the interior, $(p_{0},t_{0})$ is a
parabolically interior minimum point. We infer $\nabla f=0$ as well
as the differential inequality $\L f\le0$ at $(p_{0},t_{0})$. The
condition $\nabla f=0$ gives 
\begin{equation}
\frac{\nabla w}{w}=\frac{\nabla\psi}{\psi}\quad\text{at }(p_{0},t_{0})\;.\label{eq gb first condition}
\end{equation}
The other condition yields by virtue of (\ref{eq evol w}), (\ref{eq evol log}),
(\ref{eq evol psi}), and (\ref{eq gb first condition}) at $(p_{0},t_{0})$
\[
\begin{split}0 & \ge\L f=\L\log w-\L\log\psi\\
 & =|A|^{2}+\left|\frac{\nabla w}{w}\right|^{2}+\frac{1-\alpha}{\alpha}\,H\,w\,\psi^{-1}+\frac{b}{\alpha\,H^{\alpha-1}}\,\psi^{-1}-\left|\frac{\nabla\psi}{\psi}\right|^{2}\\
 & =|A|^{2}+\frac{1-\alpha}{\alpha}\,H\,w\,\psi^{-1}+\frac{b}{\alpha\,H^{\alpha-1}}\,\psi^{-1}\,.
\end{split}
\]
In the case $\alpha\le1$, this is impossible, and hence we must have
$t_{0}=0$, and (\ref{eq gradient bound}) follows for this case.

In the case $\alpha>1$, we use $|A|^{2}\ge\frac{1}{n}H^{2}$ and
$\psi\le a$ to obtain at $(p_{0},t_{0})$ 
\[
\psi^{-1}\,w\ge\frac{\alpha}{n\,(\alpha-1)}\,H+\frac{b}{a\,(\alpha-1)}\,H^{-\alpha}\ge\begin{cases}
\frac{\alpha}{n\,(\alpha-1)} & H\ge1\,,\\
\frac{b}{a\,(\alpha-1)} & H\le1\,.
\end{cases}
\]
This establishes (\ref{eq gradient bound}) at $(p_{0},t_{0})$, which
we noted to be sufficient.
\end{proof}

\paragraph{Lower bound on $H^{\alpha}$.}

\begin{equation}
\psi^{-1}\,H^{\alpha}\ge\begin{cases}
\inf_{t=0}\psi^{-1}\,H^{\alpha} & \alpha\le1\;,\\
\min\left\{ \inf\limits _{t=0}\psi^{-1}\,H^{\alpha},\;\frac{b}{a\,(\alpha-1)}\right\}  & \alpha>1\;.
\end{cases}\label{eq lower bound H}
\end{equation}

\begin{proof}
Again, it suffices to prove the estimate (\ref{eq lower bound H})
on time intervals $[0,T]$ where the end point $T$ is kept arbitrary.
Let $(p_{0},t_{0})$ be the minimal point of $\psi^{-1}\,H^{\alpha}$
over those times. If $t_{0}=0$, (\ref{eq lower bound H}) follows.
So let us assume $t_{0}>0$. Because $\psi$ is a localization function
(in particular it vanishes towards the boundary of the set $\{\psi>0\}$,
which we consider), we can conclude from $t_{0}>0$ that $(p_{0},t_{0})$
is a parabolically interior point. The point $(p_{0},t_{0})$ also
is a minimal point for $\log H^{\alpha}-\log\psi$. Hence, it holds
at $(p_{0},t_{0})$ 
\[
\frac{\nabla H^{\alpha}}{H^{\alpha}}=\frac{\nabla\psi}{\psi}
\]
and 
\[
\begin{split}0 & \ge\L(\log H^{\alpha}-\log\psi)=\frac{\L H^{\alpha}}{H^{\alpha}}-\frac{\L\psi}{\psi}+0\\
 & =|A|^{2}+\frac{1-\alpha}{\alpha}\,H\,w\,\psi^{-1}+\frac{b}{\alpha\,H^{\alpha-1}}\,\psi^{-1}\\
 & \ge+\frac{1-\alpha}{\alpha}\,H\,w\,\psi^{-1}+\frac{b}{\alpha\,H^{\alpha-1}}\,\psi^{-1}\;.
\end{split}
\]
For $\alpha\le1$ this is impossible, and (\ref{eq lower bound H})
follows in this case. If $\alpha>1$, we use $w\le1$ and rearrange
to 
\begin{equation}
\psi^{-1}\,H^{\alpha}\ge\frac{b}{\alpha-1}\,\psi^{-1}\ge\frac{b}{a\,(\alpha-1)}\;,
\end{equation}
and (\ref{eq lower bound H}) holds in this case, too.
\end{proof}

\paragraph{Upper bound on $H$.}

From now on, $b=0$ is possible and we will henceforth choose it that
way.

\global\long\def\uw{\underline{w}}%
 We assume a gradient bound: $w\ge2\,\uw$ on the set $\{(p,t)\colon\psi(p,t)>0\}$
where $\uw>0$ is a constant. In what follows, constants may depend
on $\uw$ and consequently depend on the gradient bound.

The following localized bound for $H$ holds: 
\begin{equation}
\psi\,H\le C(n,\alpha,a,\uw)\,\max\left\{ \sup_{t=0}\psi\,H,\;1\right\} \,.\label{eq upper bound H}
\end{equation}

\begin{proof}
Without loss of generality, we may only consider the finite time interval
$[0,T]$ for an arbitrary, but fixed $T>0$.

We define the test function 
\[
f\coloneqq\log H^{\alpha}+\alpha\log\psi-\log(w-\uw)\,.
\]
Let $(p_{0},t_{0})$ be a maximal point of $f$ over this time interval.
If $t_{0}=0$, then $f\le\sup_{t=0}f$. By exponentiation, this translates
to 
\[
\psi^{\alpha}\,H^{\alpha}\le(w-\uw)\,\sup_{t=0}\frac{\psi^{\alpha}\,H^{\alpha}}{w-\uw}\le\frac{1}{\uw}\,\sup_{t=0}\psi^{\alpha}\,H^{\alpha}\,,
\]
where we have used $\uw\le w-\uw\le1$. This shows (\ref{eq upper bound H})
in the case $t_{0}=0$.

Suppose $t_{0}>0$. Then $(p_{0},t_{0})$ is a parabolically interior
point because $\psi$ is a localization function. At this maximal
point $(p_{0},t_{0})$, there hold $\nabla f=0$ and $\L f\ge0$,
which respectively amount to 
\begin{equation}
\frac{\nabla H^{\alpha}}{H^{\alpha}}=-\alpha\,\frac{\nabla\psi}{\psi}+\frac{\nabla w}{w-\uw}\label{eq ubH first cond}
\end{equation}
and 
\begin{equation}
\begin{split}0 & \le\frac{\L H^{\alpha}}{H^{\alpha}}+\left|\frac{\nabla H^{\alpha}}{H^{\alpha}}\right|^{2}+\alpha\,\L\log\psi-\frac{\L w}{w-\uw}-\left|\frac{\nabla w}{w-\uw}\right|^{2}\\
 & =|A|^{2}+\left|\frac{\nabla H^{\alpha}}{H^{\alpha}}\right|^{2}+(\alpha-1)\,H\,w\,\psi^{-1}+\alpha\left|\frac{\nabla\psi}{\psi}\right|^{2}-\frac{w}{w-\uw}|A|^{2}-\left|\frac{\nabla w}{w-\uw}\right|^{2}\\
 & =-\frac{\uw}{w-\uw}|A|^{2}+(\alpha-1)\,H\,w\,\psi^{-1}+\left|\frac{\nabla H^{\alpha}}{H^{\alpha}}\right|^{2}+\alpha\left|\frac{\nabla\psi}{\psi}\right|^{2}-\left|\frac{\nabla w}{w-\uw}\right|^{2}\,.
\end{split}
\label{eq ubH second cond}
\end{equation}

From (\ref{eq ubH first cond}) we can deduce for $\varepsilon>0$
\begin{equation}
\left|\frac{\nabla H^{\alpha}}{H^{\alpha}}\right|^{2}\le(1+\varepsilon^{-1})\,\alpha^{2}\left|\frac{\nabla\psi}{\psi}\right|^{2}+(1+\varepsilon)\left|\frac{\nabla w}{w-\uw}\right|^{2}\,.\label{eq ubH Young}
\end{equation}
Furthermore, the following hold 
\begin{align}
|\nabla\psi|^{2} & =|\nabla X^{n+1}|^{2}\le1\,,\label{eq ubH nabla psi}\\
\nabla_{i}w & =-\nabla_{i}\nu^{n+1}=-h_{i}^{j}\,\nabla_{j}X^{n+1}\,,\\
|\nabla w|^{2} & \le|A|^{2}\,|\nabla X^{n+1}|^{2}\le|A|^{2}\,.\label{eq ubH nabla w}
\end{align}
By (\ref{eq ubH second cond}), (\ref{eq ubH Young}), (\ref{eq ubH nabla w}),
(\ref{eq ubH nabla psi}), and $w-\uw\ge\uw$, 
\[
\begin{split}0 & \le-\frac{\uw}{w-\uw}|A|^{2}+(\alpha-1)\,H\,w\,\psi^{-1}+\big(\alpha+(1+\varepsilon^{-1})\,\alpha^{2}\big)\left|\frac{\nabla\psi}{\psi}\right|^{2}+\varepsilon\left|\frac{\nabla w}{w-\uw}\right|^{2}\\
 & \le-\frac{\uw-\varepsilon/\uw}{w-\uw}\,|A|^{2}+(\alpha-1)\,H\,w\,\psi^{-1}+\big(\alpha+(1+\varepsilon^{-1})\,\alpha^{2}\big)\,\psi^{-2}\,.
\end{split}
\]
If we set $\varepsilon\coloneqq\frac{1}{2}\,\uw^{2}$ and use $|A|^{2}\ge\frac{1}{n}H^{2}$,
$w-\uw\le1$, and $w\le1$, we obtain 
\[
0\le-\frac{1}{2}\,\uw\,\frac{H^{2}}{n}+C\,H\,\psi^{-1}+C\,\psi^{-2}\,,
\]
where the constant $C>0$ depends on $\alpha$ and on $\uw$. Rearranging
and adjusting the constant, which may now also depend on $n$, yields
\[
(\psi\,H)^{2}\le C\,(\psi\,H+1)\,.
\]
It is now easy to see that $\psi\,H$ is bounded in a controlled fashion
at $(p_{0},t_{0})$, the point at which our calculations take place,
and which is the maximal point of our test function $f$. Put together
with the previous case of $t_{0}=0$, we conclude (by considering
$\exp f$) 
\[
\psi^{\alpha}\,H^{\alpha}\,(w-\uw)^{-1}\le\max\left\{ \sup_{t=0}\psi^{\alpha}\,H^{\alpha}\,(w-\uw)^{-1},\;C\right\} \,.
\]
Because of $\uw\le w-\uw\le1$, the assertion (\ref{eq upper bound H})
easily follows from here.
\end{proof}

\paragraph{Curvature bound.}

We assume the following both-sided bound on $H^{\alpha}$: 
\begin{equation}
0<c\le\frac{1}{2}\,H^{\alpha}\le c^{-1}\;.\label{eq cb control H}
\end{equation}
Then, depending on this control for $H^{\alpha}$ and the control
on the gradient, the following localized curvature bound holds: 
\begin{equation}
\psi\,|A|\le C(c,\uw,a,\alpha)\,\max\left\{ \sup_{t=0}\psi\,|A|,\,1\right\} \,.\label{eq curvature bound}
\end{equation}

\begin{proof}
Our strategy is the same as with the other estimates before. This
time, we consider the test function 
\[
f\coloneqq\log|A|^{2}+2\log\psi-\beta\,\log(H^{\alpha}-c)\,,
\]
where we are going to choose $\beta>0$ later.

First of all, we record an inequality we are going to use: 
\begin{align}
 & |\nabla|A|^{2}|^{2}=|2\,h^{ij}\,\nabla h_{ij}|^{2}\le4\,|A|^{2}\,|\nabla A|^{2}\nonumber \\
\Longrightarrow\quad & -2\,\frac{|\nabla A|^{2}}{|A|^{2}}\le-\frac{1}{2}\,\left|\frac{\nabla|A|^{2}}{|A|^{2}}\right|^{2}\,.\label{eq cb nabla A}
\end{align}
By (\ref{eq evol |A|^2}), (\ref{eq evol log}), and (\ref{eq cb nabla A}),
\[
\begin{split}\L\log|A|^{2} & =\frac{\L|A|^{2}}{|A|^{2}}+\left|\frac{\nabla|A|^{2}}{|A|^{2}}\right|^{2}\\
 & =-2\,\frac{|\nabla A|^{2}}{|A|^{2}}+2\,(\alpha-1)\,H^{-1}\,h^{kl}\,\nabla_{k}H\,\nabla_{l}H\,|A|^{-2}\\
 & \quad+2\,|A|^{2}+2\,\frac{1-\alpha}{\alpha}\,H\,h_{j}^{i}\,h_{k}^{j}\,h_{i}^{k}\,|A|^{-2}+\left|\frac{\nabla|A|^{2}}{|A|^{2}}\right|^{2}\\
 & \le\frac{1}{2}\left|\frac{\nabla|A|^{2}}{|A|^{2}}\right|^{2}+2\,|\alpha-1|\,(H\,|A|)^{-1}\,|\nabla H|^{2}+2\,|A|^{2}+2\,\frac{|\alpha-1|}{\alpha}\,H\,|A|\,.
\end{split}
\]

In a parabolically interior, first maximal point $(p_{0},t_{0})$
of $f$, $\L f\ge0$ holds: 
\begin{equation}
\begin{split}0 & \le\L f=\L\log|A|^{2}+2\,\L\log\psi-\beta\L\log(H^{\alpha}-c)\\
 & \le\frac{1}{2}\left|\frac{\nabla|A|^{2}}{|A|^{2}}\right|^{2}+2\,|\alpha-1|\,(H\,|A|)^{-1}\,|\nabla H|^{2}+2\,|A|^{2}+2\,\frac{|\alpha-1|}{\alpha}\,H\,|A|\\
 & \quad+2\,\frac{|\alpha-1|}{\alpha}\,H\,w\,\psi^{-1}+2\left|\frac{\nabla\psi}{\psi}\right|^{2}-\beta\,\frac{H^{\alpha}}{H^{\alpha}-c}\,|A|^{2}-\beta\,\left|\frac{\nabla H^{\alpha}}{H^{\alpha}-c}\right|^{2}\,.
\end{split}
\label{eq cb Lf ge 0}
\end{equation}
The vanishing of $\nabla f$ at $(p_{0},t_{0})$ yields 
\begin{equation}
\begin{split}\left|\frac{\nabla|A|^{2}}{|A|^{2}}\right|^{2} & =\left|\beta\,\frac{\nabla H^{\alpha}}{H^{\alpha}-c}-2\,\frac{\nabla\psi}{\psi}\right|^{2}\\
 & \le(1+\varepsilon)\,\beta^{2}\left|\frac{\nabla H^{\alpha}}{H^{\alpha}-c}\right|^{2}+(1+\varepsilon^{-1})\,4\left|\frac{\nabla\psi}{\psi}\right|^{2}
\end{split}
\label{eq cb nabla f eq 0}
\end{equation}
for an arbitrary $\varepsilon>0$ to be chosen. We substitute (\ref{eq cb nabla f eq 0})
into (\ref{eq cb Lf ge 0}). Let us for the moment assume $H\le\delta\,|A|$
at $(p_{0},t_{0})$ for a (small) constant $\delta>0$ to be chosen.
We obtain 
\begin{equation}
\begin{split}0 & \le\left((1+\varepsilon)\frac{\beta^{2}}{2}-\beta\right)\left|\frac{\nabla H^{\alpha}}{H^{\alpha}-c}\right|^{2}+2\,|\alpha-1|\,\delta\left|\frac{\nabla H}{H}\right|^{2}\\
 & \quad+\left(2-\beta\,\frac{H^{\alpha}}{H^{\alpha}-c}+2\,\frac{|\alpha-1|}{\alpha}\,\delta\right)|A|^{2}\\
 & \quad+2\,\frac{|\alpha-1|}{\alpha}\,H\,w\,\psi^{-1}+\big(2+2\,(1+\varepsilon^{-1})\big)\left|\frac{\nabla\psi}{\psi}\right|^{2}\,.
\end{split}
\label{eq cb after Young}
\end{equation}
We use 
\[
\left|\frac{\nabla H}{H}\right|^{2}=\left|\nabla\log H\right|^{2}=\left|\alpha^{-1}\nabla\log H^{\alpha}\right|^{2}=\frac{1}{\alpha^{2}}\left|\frac{\nabla H^{\alpha}}{H^{\alpha}}\right|^{2}\le\frac{1}{\alpha^{2}}\left|\frac{\nabla H^{\alpha}}{H^{\alpha}-c}\right|^{2}
\]
and (\ref{eq ubH nabla psi}) to obtain from (\ref{eq cb after Young})
with certain controlled constants 
\begin{equation}
\begin{split}0 & \le\left((1+\varepsilon)\,\frac{\beta^{2}}{2}-\beta+2\,\frac{|\alpha-1|}{\alpha^{2}}\,\delta\right)\left|\frac{\nabla H^{\alpha}}{H^{\alpha}-c}\right|^{2}\\
 & \quad+\left(2-\beta\,\frac{H^{\alpha}}{H^{\alpha}-c}+2\,\frac{|\alpha-1|}{\alpha}\,\delta\right)|A|^{2}\\
 & \quad+C\,\psi^{-1}+C(\varepsilon)\,\psi^{-2}\,.
\end{split}
\label{eq cb up to coef}
\end{equation}
We choose $\beta=2-\frac{c^{2}}{2}$. For sufficiently small (but
controlled) $\varepsilon,\delta>0$, then hold 
\begin{equation}
\begin{split}(1+\varepsilon)\,\frac{\beta^{2}}{2}-\beta+2\,\frac{|\alpha-1|}{\alpha^{2}}\,\delta & =\frac{1}{2}\left(4-2\,c^{2}+\frac{c^{4}}{4}\right)-\left(2-\frac{c^{2}}{2}\right)+\varepsilon\,\frac{\beta^{2}}{2}+2\,\frac{|\alpha-1|}{\alpha^{2}}\,\delta\\
 & =-\frac{c^{2}}{2}+\frac{c^{4}}{4}+\varepsilon\,\frac{\beta^{2}}{2}+2\,\frac{|\alpha-1|}{\alpha^{2}}\,\delta\\
 & \le0\qquad\text{(note that \ensuremath{c\le1})}
\end{split}
\label{eq cb 1. coef}
\end{equation}
and 
\begin{equation}
\begin{split}2-\beta\,\frac{H^{\alpha}}{H^{\alpha}-c}+2\,\frac{|\alpha-1|}{\alpha}\,\delta & =\frac{2\,(H^{\alpha}-c)-\left(2-\frac{c^{2}}{2}\right)H^{\alpha}}{H^{\alpha}-c}+2\,\frac{|\alpha-1|}{\alpha}\,\delta\\
 & =\frac{-2\,c+\frac{c^{2}}{2}\,H^{\alpha}}{H^{\alpha}-c}+2\,\frac{|\alpha-1|}{\alpha}\,\delta\\
 & \stackrel{\text{\eqref{eq cb control H}}}{\le}\frac{-c}{H^{\alpha}-c}+2\,\frac{|\alpha-1|}{\alpha}\,\delta\\
 & \stackrel{\text{\eqref{eq cb control H}}}{\le}\frac{-c^{2}}{2}+2\,\frac{|\alpha-1|}{\alpha}\,\delta\\
 & \le\frac{-c^{2}}{4}\;.
\end{split}
\label{eq cb 2. coef}
\end{equation}
We substitute (\ref{eq cb 1. coef}) and (\ref{eq cb 2. coef}) into
(\ref{eq cb up to coef}) and multiply by $\psi^{2}$: 
\[
0\le-\frac{c^{2}}{4}\,|A|^{2}\,\psi^{2}+C\,\psi+C(\varepsilon)\,.
\]
The localization function $\psi$ is bounded by $a$ and $\varepsilon>0$
is a controlled quantity. Thus, there is a controlled constant such
that 
\[
\psi\,|A|\le C\qquad\text{at }(p_{0},t_{0})\,.
\]

We have assumed $H\le\delta\,|A|$ at $(p_{0},t_{0})$. Now we rectify
this. If $|A|<\frac{H}{\delta}$, then we still find $\psi\,|A|\le C$
by the control on $H$, $\delta$, and $\psi$. Because of (\ref{eq cb control H})
it thus holds 
\[
f(p_{0},t_{0})=\left.\left(\log|A|^{2}+2\,\log\psi-\beta\,\log(H^{\alpha}-c)\right)\right|_{(p_{0},t_{0})}\le C\,.
\]
At any point $(p,t)$ the function $f$ is bounded by this constant
or by its initial values. Through consideration of $\exp\frac{f}{2}$,
we obtain 
\[
\psi\,|A|\le(H^{\alpha}-c)^{\beta/2}\,\max\left\{ \sup_{t=0}\frac{\psi\,|A|}{(H^{\alpha}-c)^{\beta/2}},\;C\right\} \le C\,\max\left\{ \sup_{t=0}\psi\,|A|,\;1\right\} ,
\]
which is the asserted inequality (\ref{eq curvature bound}).
\end{proof}

\section{Complete hypersurfaces\label{sec:Complete-hypersurfaces}}

\begin{thmbis}{thm intro}\label{thm main} Let $\alpha>0$. Let $\Omega_{0}\subset\mathbb{R}^{n}$
be open. Let $u_{0}\colon\Omega_{0}\to\mathbb{R}$ be smooth and such
that $\graph u_{0}$ is of positive mean curvature $H[u_{0}]>0$.
Furthermore, we suppose that the sets $\{x:u_{0}(x)\le a\}$ are compact
for any $a\in\mathbb{R}$.

Then, there exists a relatively open set $\Omega\subset\mathbb{R}^{n}\times[0,\infty)$
compatible with the $\Omega_{0}$ from above ($\Omega\cap\left(\R^{n}\times\{0\}\right)=\Omega_{0}$),
and there exists a continuous function $u\colon\Omega\to\mathbb{R}$
which is smooth on $\Omega\setminus\left(\Omega_{0}\times\{0\}\right)$,
and such that $u(\cdot,0)=u_{0}$, and that $u$ is a solution of
the graphical $H^{\alpha}$-flow (\ref{eq H^alpha-flow graph}). Moreover,
$u$ fulfills the following maximality condition: There is a continuous
function $\overline{u}\colon\mathbb{R}^{n}\times[0,\infty)\to\overline{\mathbb{R}}$
such that $\{(x,t):\overline{u}(x,t)\in\mathbb{R}\}=\Omega$ and $\overline{u}|_{\Omega}=u$.\end{thmbis}
\begin{rem}
The maximality condition implies the completeness at every fixed time.
On the other hand it implies that the flow is maximal in the sense
that it is not arbitrarily stopped or runs into any singularities.
\end{rem}

\begin{proof}
Let $(a_{k})\subset\mathbb{R}$ be a sequence of heights such that
$a_{k}\to\infty$ and such that $Q_{k}\coloneqq\{u_{0}<a_{k}\}$ are
smooth, open, and bounded domains. (The existence of such a sequence
is guaranteed by Sard's Theorem.) For fixed $k\in\mathbb{N}$, we
consider the initial boundary value problem (\ref{eq ap ap}) with
domain $Q_{k}$ and initial datum $u_{0}|_{Q_{k}}$. (We observe that
$u_{0}|_{\partial Q_{k}}\equiv a_{k}$ because if $x\in\partial Q_{k}$
and $Q_{k}\ni x_{i}\to x\in\partial Q_{k}$, then, by the compactness
of $\{y:u_{0}(y)\le a_{k}\}$, the boundary point is in this set:
$x\in\{y:u_{0}(y)\le a_{k}\}\subset\Omega_{0}$. Therefore, $u_{0}(x)$
is well-defined and $u_{0}(x)=a_{k}$ must hold by the continuity
of $u_{0}$.) 
\global\long\def\ta{\tilde{\alpha}}%
 The exponent function will be denoted $\ta$ here. It is chosen such
that $\ta\equiv\alpha$ on $\{u_{0}\le a_{k}-1\}$ and such that $\ta(x)$
is always between $\alpha$ and $1$ (including those values). The
cut-off function $\xi$ in (\ref{eq ap ap}) is chosen according to
the assumptions on it, as is the constant $c$. Let $v_{k}\colon Q_{k}\times[0,\infty)\to\mathbb{R}$
be the solution from Lemma \ref{lem ap}. We extend $v_{k}$ to $\mathbb{R}^{n}\times[0,\infty)$
by setting $\overline{v}_{k}(x,t)=v_{k}(x,t)$ if $(x,t)\in Q_{k}\times[0,\infty)$
and $\overline{v}_{k}(x,t)=a_{k}+ct$ if $(x,t)\notin Q_{k}\times[0,\infty)$.
Then $\overline{v}_{k}$ is a continuous function on $\mathbb{R}^{n}\times[0,\infty)$.

Because the solutions of (\ref{eq ap ap}) are growing in time, the
$t$-dependent set $\{u(\cdot,t)\le a_{k}-1\}$ is shrinking with
$t$. For this reason, $\big(\graph\overline{v}_{k}(\cdot,t)\big)_{t\ge0}$
moves by the $H^{\alpha}$-flow below the height $a_{k}-1$. This
makes the a priori bounds from section \ref{sec:A-priori-estimates}
applicable. These yield estimates on $\nabla\overline{v}_{k}$ and
$\nabla^{2}\overline{v}_{k}$ at points $(x,t)$ where $\overline{v}_{k}(x,t)<a$
and $t<t_{*}$ for arbitrary $a<a_{k}-2$ and $t_{*}>0$. The estimates
depend on $a$ and $t_{*}$ but not on $k$. Estimates on $\partial_{t}\overline{v}_{k}$
of the same kind are obtained from the equation of graphical $H^{\alpha}$-flow
once we have the estimates on $\nabla\overline{v}_{k}$ and $\nabla^{2}\overline{v}_{k}$.
This is enough to apply the Arzelà-Ascoli argument of \cite{SS}.
It gives a subsequence of $\overline{v}_{k}$ and a continuous function
$\overline{u}\colon\mathbb{R}^{n}\times[0,T)\to\overline{\mathbb{R}}$
such that $\overline{v}_{k_{l}}\to\overline{u}$ pointwise and locally
uniformly on $\Omega\coloneqq\{\overline{u}\in\mathbb{R}\}$.

In addition to the local $C^{2;1}$-estimates mentioned above, section
\ref{sec:A-priori-estimates} also supplies us with the estimates
(\ref{eq lower bound H}) and (\ref{eq upper bound H}) which give
both-sided control on $H^{\alpha}$. This lets us locally control
the parabolicity constants. Here, ``local'' means on sets $\{(x,t)\colon\overline{v}_{k}(x,t)<a,\;t<t_{*}\}$.
As in the proof of Lemma \ref{lem ap}, local estimates on all derivatives
of the $\overline{v}_{k}$ uniform in $k$ now follow from general
theory. These estimates are local in time also for $t\to0$: Uniform
estimates are only obtained for $t>\varepsilon$ and the constants
then depend on $\varepsilon$. These estimates are still sufficient
to show that the locally uniform convergence $\overline{v}_{k_{l}}\to\overline{u}$
on $\Omega$ is a locally smooth convergence on $\Omega\setminus\left(\Omega_{0}\times\{0\}\right)$.

We set $u\coloneqq\overline{u}|_{\Omega}$. It is quite clear that
$u(x,0)=u_{0}(x)$ for $x\in\Omega_{0}$ and that $\overline{u}(x,0)=\infty$
for $x\notin\Omega_{0}$. That $u$ is a solution of the graphical
$H^{\alpha}$-flow is fairly easy to see too: Because $\big(\graph\overline{v}_{k}(\cdot,t)\big)_{t\ge0}$
moves by $H^{\alpha}$-flow below the height $a_{k}-1$ and $\overline{v}_{k_{l}}$
converges locally smoothly to $u$ on $\Omega\setminus\left(\Omega_{0}\times\{0\}\right)$,
$u$ must solve the graphical $H^{\alpha}$-flow on $\Omega\setminus\left(\Omega_{0}\times\{0\}\right)$. 
\end{proof}


\begin{thebibliography}{10}
\bibitem{AS}\emph{Roberta Alessandroni} and \emph{Carlo Sinestrari},
Evolution of convex entire graphs by curvature flows, Geometric Flows
\textbf{1} (2015), 1111--1125.

\bibitem{AW}\emph{Ben Andrews} and \emph{Yong Wei}, Volume preserving
flow by powers of $k$-th mean curvature, J. Differential Geom. \textbf{117(2)}
(2021), 193-222 .

\bibitem{BCD}\emph{Simon Brendle}, \emph{Kyeongsu Choi} and \emph{Panagiota
Daskalopoulos}, Asymptotic behavior of flow by powers of the Gaussian
curvature, Acta Math. \textbf{219} (2017), 1--16.

\bibitem{CD}\emph{Kyeongsu Choi} and \emph{Panagiota Daskalopoulos},
The $Q_{k}$ flow on complete non-compact graphs, arXiv:1603.03453
{[}math.DG{]} (2016).

\bibitem{CDKL}\emph{K.\ Choi}, \emph{P.\ Daskalopoulos}, \emph{L.\ Kim}
and \emph{K.A.\ Lee}, The evolution of complete non-compact graphs
by powers of Gauss curvature, J.\ reine angew.\ Math.\ \textbf{757}
(2019), 131--158.

\bibitem{Gerhardt}\emph{Claus Gerhardt}, Curvature Problems, Series
in Geometry and Topology Vol.\  \textbf{39}, International Press
(2006).

\bibitem{EH}\emph{Klaus Ecker} and \emph{Gerhard Huisken}, Interior
estimates for hypersurfaces moving by mean curvature, Invent.\ Math.\ \textbf{105.3}
(1991), 547--569.

\bibitem{Franzen}\emph{Martin Franzen}, Entire Graphs Evolvin by
Powers of the Mean Curvature, arXiv:1112.4359 {[}math.DG{]} (2011).

\bibitem{Holland}\emph{James Holland}, Interior Estimates for Hypersurfaces
Evolving by Their k-th Weingarten Curvature, and Some Applications,
Indiana University Mathematics Journal \textbf{63.5} (2014), 1281--1310.

\bibitem{Huisken}\emph{Gerhard Huisken}, Flow by mean curvature of
convex surfaces into spheres, J.\ Differential Geom.\ \textbf{20.1}
(1984), 237--266.

\bibitem{LSU}\emph{O.\ Lady\v{z}enskaja}, \emph{V.\ Solonnikov}
and \emph{N.\ Ural'ceva}, Linear and Quasilinear Equations of Parabolic
Type, Translated from the Russian by S.\ Smith, Translations of Mathematical
Monographs, Vol.\ \textbf{23}, Providence, R.I.:American Mathematical
Society (1968).

\bibitem{LiLv}\emph{Guanghan Li} and \emph{Yusha Lv}, Evolution of
complete noncompact graphs by powers of curvature function, arXiv:1901.04099
{[}math.DG{]} (2019). 

\bibitem{SS}\emph{Mariel Sáez Trumper} and \emph{Oliver Schnürer},
Mean curvature flow without singularities, J.\ Differential Geom.\ \textbf{97.3}
(2014), 545--570. 

\bibitem{Schulze-Evolution}\emph{Felix Schulze}, Evolution of convex
hypersurfaces by powers of the mean curvature, Mathematische Zeitschrift
\textbf{251.4} (2005), 721--733.

\bibitem{Schulze-Convexity}\emph{Felix Schulze}, Convexity estimates
for flows by powers of the mean curvature, Ann. Scuola Norm. Sup.
Pisa Cl. Sci. (5) Vol. V (2006), 261--277.

\bibitem{TW}\emph{Guji Tian} and \emph{Xu-Jia Wang}, A Priori Estimates
for Fully Nonlinear Parabolic Equations, International Mathematics
Research Notices (2013).

\bibitem{Xiao}\emph{Ling Xiao}, General curvature flow without singularities,
arXiv:1604.05743 {[}math.DG{]} (2016).
\end{thebibliography}
\end{document}